\documentclass[11pt,a4paper]{article}

\usepackage[margin=30mm]{geometry}
\usepackage{amsmath,amssymb,amsthm,mathtools,amsfonts}
\usepackage{bm}
\usepackage{hyperref}
\usepackage{enumitem}
\usepackage{microtype}

\hypersetup{
  colorlinks=true,
  linkcolor=blue, %black,
  citecolor=blue,
  urlcolor=blue
}

\newcommand{\alg}[1]{\mathfrak{#1}}
\newcommand{\CC}{\mathbb{C}}
\newcommand{\ZZ}{\mathbb{Z}}

\newcommand{\De}{\Delta}

\newtheorem{theorem}{Theorem}[section]
\newtheorem{proposition}[theorem]{Proposition}
\newtheorem{lemma}[theorem]{Lemma}

\theoremstyle{definition}
\newtheorem{definition}[theorem]{Definition}
\theoremstyle{remark}
\newtheorem{remark}[theorem]{Remark}
\newtheorem{example}{Example}[section]

\begin{document}
\title{\Large\bf Smith Algebra from 
\texorpdfstring{$\mathfrak{sl}_2$}{sl2} 
Tensor Product Branching \\
via the Mickelsson--Zhelobenko Algebra}

\author{Kohei Hasegawa$^\ast$
\and Takuya Matsumoto$^\dagger$
\and Akito Takashima$^\ddagger$}

\date{} % Start November 2024
\maketitle

\noindent
{\it
${}^{\dagger}$
Department of Applied Physics, 
School of Engineering, 
University of Fukui, 
$3$-$9$-$1$ Bunkyo, Fukui-shi, Fukui $910$-$8507$, Japan
}\\
\noindent
{\it  ${}^{\ast, \dagger, \ddagger}$
Fundamental Engineering for Knowledge-Based Society, 
Graduate School of Engineering, 
University of Fukui, 
$3$-$9$-$1$ Bunkyo, Fukui-shi, Fukui $910$-$8507$, Japan}
%\\
%{\it
%${}^{\dagger}$
%Department of Human and Artificial Intelligent Systems (HART), 
%Graduate School of Engineering, University of Fukui, 
%$3$-$9$-$1$ Bunkyo, Fukui-shi, Fukui $910$-$8507$, Japan
%}
\\
~\\
E-mail: {\tt \{$^\ast$mf250843, $^\dagger$takuyama, $^\ddagger$mf260880\}@u-fukui.ac.jp} 

\vspace{20mm}

\setcounter{footnote}{0}
\setcounter{page}{0}
\thispagestyle{empty}

%\newpage
%
%\tableofcontents

\begin{abstract}
Let $L_n$ be the finite-dimensional irreducible
$\mathfrak{sl}_2$-module of highest weight $n$.
We study the space of highest-weight vectors in
$L_n\otimes L_m$ for the diagonal
$\mathfrak{sl}_2$-action. The extremal projector
constructs these singular vectors and gives operators
connecting the highest-weight lines in the
Clebsch--Gordan decomposition. We interpret these
operators in the Mickelsson--Zhelobenko algebra of
the diagonal pair
$(\mathfrak{sl}_2\oplus\mathfrak{sl}_2,
\mathfrak{sl}_2^{\mathrm{diag}})$.
A rescaling of their quadratic relation yields the
cubic Smith relation $[E,F]=2(H^3-cH)$.
This cubic algebra is also known as a Higgs algebra
in the physics literature.
We express both the central parameter $c$ and the
Smith Casimir in terms of the quadratic
$\mathfrak{sl}_2$ Casimir elements of the two tensor
factors. 
The singular subspace becomes an irreducible
Smith module of dimension $\min\{n,m\}+1$,
with explicit formulas for the actions of
the raising and lowering operators $E$ and $F$.
Finally, comparison with finite-dimensional simple
Smith modules gives a necessary and sufficient
condition for a module to arise from this tensor-product
construction.
\end{abstract}

\newpage

\tableofcontents

\section{Introduction}

Let $L_n$ denote the irreducible $\mathfrak{sl}_2$-module
of highest weight $n\in\mathbb Z_{\geq0}$. For $n\geq m$,
the Clebsch--Gordan decomposition is
\[
L_n\otimes L_m
\simeq
\bigoplus_{r=0}^{m}L_{n+m-2r}
\]
under the diagonal $\mathfrak{sl}_2$-action. 
Each summand contributes one highest-weight line to
the singular subspace
\[
(L_n\otimes L_m)^+
=
\ker(e\otimes1+1\otimes e),
\]
where $e$ is the raising operator of $\mathfrak{sl}_2$.
We study the algebraic structure on this subspace:
how its highest-weight lines can be constructed and
how operators can move between them.

The {\it extremal projector} provides a systematic way to
construct highest-weight vectors \cite{AST,Z}. On the
finite-dimensional tensor product, its diagonal action
projects suitable vectors onto the singular subspace.
Projecting $e\otimes1$ and $f\otimes1$ produces operators
that connect singular vectors of adjacent highest weights.
The algebraic setting for these operators is the
{\it Mickelsson--Zhelobenko (MZ) algebra} associated with
\[
\bigl(\,\mathfrak{sl}_2\oplus\mathfrak{sl}_2\,,\;
\mathfrak{sl}_2^{\mathrm{diag}}\,\bigr).
\]
The construction goes back to Mickelsson's step algebras
\cite{Mick}, which encode operators acting on spaces
of highest-weight vectors.
For the diagonal pair considered here, the MZ algebra
is a rank-one diagonal reduction algebra \cite{KO}.

Our first main result is the appearance of a cubic
{\it Smith algebra} in this diagonal MZ algebra.
The projected raising and lowering operators satisfy
a quadratic relation with rational coefficients.
A suitable rescaling of these operators turns it
into the commutation relations 
\[
[H,E]=2E,\qquad [H,F]=-2F,\qquad
[E,F]=2(H^3-cH),
\]
where $c=T_1+T_2+1$ and $T_1,T_2$ are the Casimir
elements of the two tensor factors
(Theorem~\ref{thm:cubic}).
This gives a realization of the cubic Smith relations
\cite{S}; relations of this type also occur in the
{\it Higgs algebra} \cite{Hig,Eel}.

The realization also identifies the central parameters
of the Smith algebra. Besides
$c=T_1+T_2+1$, the {\it Smith Casimir} $\Omega$ 
has the particularly simple expression
\[
\Omega
=\frac14\bigl(-(T_1-T_2)^2
              +2(T_1+T_2)+3\bigr)
\]
(Proposition~\ref{prop:smith-central-parameters}).
Thus both central parameters are determined by the
Casimir elements of the two tensor factors. This
identification makes it possible to compare the Smith
modules arising from tensor products with abstract
highest-weight modules.

The algebraic realization has a concrete
representation-theoretic counterpart. For every
$n,m\geq0$, the singular subspace
$(L_n\otimes L_m)^+$ is an irreducible module for the
cubic Smith relations, of dimension
$\min\{n,m\}+1$ 
(Theorem~\ref{thm:smith-tensor-module}).
The theorem gives the actions of $E,F,H$ on a basis
of singular vectors, as well as the values of the
central parameters. In this representation, 
the Smith raising and lowering operators move between
the highest-weight vectors of successive summands
in the Clebsch--Gordan decomposition.

We then compare these modules with the
finite-dimensional simple modules of the abstract
cubic Smith algebra \cite{S}. Let $d$ be the dimension
of such a module, $\lambda$ its highest weight, and
$c$ its central parameter, and put
$\mu=\lambda-d+1$. Proposition~\ref{prop:smith-tensor-range}
shows that it arises from $(L_n\otimes L_m)^+$ if and
only if
\[
\mu\in\mathbb Z_{>0},\qquad
\mu\geq d,\qquad
c+1=\mu^2+d^2.
\]
After ordering the factors so that $n\geq m$, they are
uniquely determined by
\[
n=\mu-1,\qquad m=d-1.
\]

The paper is organized as follows. 
Section~\ref{sec:tensor}
reviews the diagonal action on tensor products.
Section~\ref{sec:projector} constructs the singular
vectors and the operators connecting them.
Section~\ref{sec:mz} introduces the diagonal MZ algebra
and derives its quadratic relation.
Section~\ref{sec:smith} obtains the cubic Smith relations
and identifies their central parameters.
Section~\ref{sec:representation} describes the resulting
modules and determines the range of the tensor-product
construction. 
Appendix~\ref{app:explicit} gives an explicit formula
for the singular vectors. Appendix~\ref{app:MZgen}
records representatives of the MZ and Smith generators,
and Appendix~\ref{app:QG} compares the cubic relation
formally with that of $U_q(\mathfrak{sl}_2)$.

\section{Tensor products and singular vectors}
\label{sec:tensor}

\subsection{The diagonal action on tensor products}

We work over $\mathbb C$.  Let $\mathfrak{sl}_2$ be the Lie algebra
with generators $e,f,h$ and relations
\[
[h,e]=2e,\qquad [h,f]=-2f,\qquad [e,f]=h.
\]
For $n\in\mathbb Z_{\geq0}$, let $L_n$ be the 
$(n+1)$-dimensional 
irreducible
$\mathfrak{sl}_2$-module of highest weight $n$, and choose a
highest-weight vector $1_n\in L_n$ such that
\[
e1_n=0,\qquad h1_n=n1_n.
\]
We use the basis
\[
v_k^{(n)}:=f^k1_n,\qquad 0\leq k\leq n,
\]
with the convention $v_{-1}^{(n)}=v_{n+1}^{(n)}=0$.
The action of $\mathfrak{sl}_2$ on this basis is
\[
e\,v_k^{(n)}=k(n-k+1)v_{k-1}^{(n)},\qquad
f\,v_k^{(n)}=v_{k+1}^{(n)},\qquad
h\,v_k^{(n)}=(n-2k)v_k^{(n)}.
\]

For $n,m\in\mathbb Z_{\geq0}$, the tensor product
$V=L_n\otimes L_m$ has dimension $(n+1)(m+1)$ and basis
\[
\bigl\{v_i^{(n)}\otimes v_j^{(m)}
       \mid 0\leq i\leq n,\ 0\leq j\leq m\bigr\}.
\]
It carries the diagonal $\mathfrak{sl}_2$-action induced by the
coproduct
\[
\Delta:U(\mathfrak{sl}_2)\longrightarrow
U(\mathfrak{sl}_2)\otimes U(\mathfrak{sl}_2),
\qquad
\Delta(x)=x\otimes1+1\otimes x
\quad (x\in \alg{sl}_2).
\]
For any $a,b\in U(\alg{sl}_2)$, 
it satisfies $\De(ab)=\De(a)\De(b)$. 
The action is explicitly given by 
\begin{align*}
\Delta(e)\bigl(v_i^{(n)}\otimes v_j^{(m)}\bigr)
&=i(n-i+1)v_{i-1}^{(n)}\otimes v_j^{(m)}
  +j(m-j+1)v_i^{(n)}\otimes v_{j-1}^{(m)},\\
\Delta(f)\bigl(v_i^{(n)}\otimes v_j^{(m)}\bigr)
&=v_{i+1}^{(n)}\otimes v_j^{(m)}
  +v_i^{(n)}\otimes v_{j+1}^{(m)},\\
\Delta(h)\bigl(v_i^{(n)}\otimes v_j^{(m)}\bigr)
&=(n+m-2i-2j)v_i^{(n)}\otimes v_j^{(m)}.
\end{align*}

\subsection{Singular vectors and the branching rule}

Let $V=L_n\otimes L_m$ with the diagonal
$\mathfrak{sl}_2$-action defined above.  
\begin{definition}
A nonzero vector
$v\in V$ is called a {\it singular vector} of weight $\lambda$
if
\[
\Delta(e)v=0,\qquad \Delta(h)v=\lambda v.
\]
\end{definition}
Such a vector is a highest-weight vector for the diagonal
$\mathfrak{sl}_2$-action.  Since $V$ is finite-dimensional,
$\lambda$ is a nonnegative integer, and the submodule
generated by $v$ is isomorphic to $L_\lambda$.
We write
\[
V^+:=\ker\Delta(e).
\]
Thus $V^+$ is the space spanned by the singular vectors;
an arbitrary element of $V^+$ need not be a weight vector.

Assume henceforth that $n\geq m$, and put $N=n+m$.
The Clebsch--Gordan rule gives the decomposition
\begin{equation}
\label{eq:CG}
V=L_n\otimes L_m
\simeq \bigoplus_{r=0}^{m}L_{N-2r}
\end{equation}
as a module for the diagonal $\mathfrak{sl}_2$.

For $\lambda\in\mathbb Z$, let $V_\lambda$ denote the
$\Delta(h)$-weight space of weight $\lambda$, and set
\[
V^+_\lambda:=V^+\cap V_\lambda.
\]
Each summand $L_{N-2r}$ in \eqref{eq:CG} has a
one-dimensional highest-weight space.  Since the highest
weights $N-2r$ are distinct, we obtain
\begin{equation}
\label{eq:singular-decomposition}
V^+=\bigoplus_{r=0}^{m}V^+_{N-2r},
\qquad
\dim V^+_{N-2r}=1
\quad (0\leq r\leq m).
\end{equation}
In particular, $\dim V^+=m+1$.

The decomposition \eqref{eq:singular-decomposition} identifies the
singular lines corresponding to the irreducible summands of
\eqref{eq:CG}.  We now construct vectors on these lines using the
extremal projector and study operators connecting adjacent lines.

\section{The extremal projector and singular vectors}
\label{sec:projector}

\subsection{The extremal projector}

The {\it extremal projector}, originating in the work of
Asherova--Smirnov--Tolstoy and developed by Zhelobenko,
is a standard tool for constructing highest-weight vectors
and operators acting on them; see \cite{AST,Z}.
For its role in the broader representation-theoretic setting
related to Yangians and Mickelsson algebras, see also
Molev \cite{Mol}.
In the case of $\mathfrak{sl}_2$, it takes the following form:
\begin{align}
\label{eq:p}
p&=\sum_{k=0}^{\infty}
f^k e^k\frac{(-1)^k }
{k!(h+2)(h+3)\cdots(h+k+1)}
\\
&=
1-fe\frac{1}{h+2}+f^2e^2\frac{1}{2(h+2)(h+3)}
-f^3e^3\frac{1}{6(h+2)(h+3)(h+4)}+\cdots. 
\notag 
\end{align}

The formula is interpreted after localizing the Cartan part.  Namely,
with $\mathfrak h=\CC h$,
let $R(\alg{h})$ denote the field of fractions of 
the commutative algebra 
$U(\alg{h})$ and consider the extension $U'(\alg{sl}_2)$ 
of the universal enveloping algebra 
$U(\alg{sl}_2)$ defined by 
\[
U'(\mathfrak{sl}_2)
=
U(\mathfrak{sl}_2)\otimes_{U(\mathfrak h)}R(\mathfrak h). 
\]
Introduce the vector space $F_m(\alg{sl}_2)$ of 
formal series of monomials 
$
f^ke^l 
$
of weight $m$ with coefficients $R(\alg h)$, where $2(l-k)=m$. 
Define $F(\alg{sl}_2)$ as the direct sum 
\[
F(\alg{sl}_2)=\bigoplus_m F_m(\alg{sl}_2). 
\]
That is, the elements of $F(\alg{sl}_2)$ are finite sums
$\sum_m a_m$ with $a_m\in F_m(\alg{sl}_2)$. 
It can be shown that $F(\alg{sl}_2)$ is an algebra with respect to 
the natural multiplication of formal series \cite{Z}. 
The extremal projector \eqref{eq:p} is algebraically understood 
as an element in $F(\alg{sl}_2)$.

On a finite-dimensional module the series has only finitely
many potentially nonzero terms on each vector. This does not
by itself make its rational coefficients evaluable at every
weight; its relation to a globally defined projection is
given in Proposition~\ref{prop:projector-on-modules}.

\begin{proposition}
\label{prop:projector-identities}
The extremal projector satisfies
\[
ep=0,\qquad pf=0,\qquad [h,p]=0,\qquad p^2=p.
\]
Moreover, $pv=v$ whenever $ev=0$ and the action of $p$ on 
a vector $v$ is defined.
\end{proposition}

\begin{proof}
The identities $ep=0$ and $pf=0$ follow by commuting $e$ and $f$,
respectively, through the series \eqref{eq:p}: the resulting terms
cancel in consecutive pairs.  Each summand in \eqref{eq:p} has
weight zero, giving $[h,p]=0$.
The last equation is equivalent to $p(1-p)=0$. By the expression
\eqref{eq:p} and the property $pf=0$, it can be shown that 
\begin{align}
p(1-p)&=pf
\left(\sum_{k=1}^{\infty}
f^{k-1} e^k\frac{(-1)^{k+1}}
{k!(h+2)(h+3)\cdots(h+k+1)}\right)=0. 
\end{align}
If the action of $p$ on $v$ is defined and $ev=0$, 
all terms of \eqref{eq:p} except the $k=0$ term
annihilate $v$, so $pv=v$.  
This completes the proof. 
\end{proof}

The identities in Proposition~\ref{prop:projector-identities}
also characterize the projector: it acts as the identity on
highest-weight vectors and annihilates vectors obtained by
applying $f$.  Conversely, requiring $ep=0$ and a constant
term equal to $1$ determines the coefficients of a series of
the form $\sum_{k\geq0} f^k e^k a_k(h)$ recursively, yielding
\eqref{eq:p}.

We now describe its meaning on finite-dimensional modules.
Let $W$ be a finite-dimensional $\mathfrak{sl}_2$-module.
On each irreducible summand $L_\ell$ of $W$, the space of 
the highest-weight vectors is complementary 
to $fL_\ell$.  Taking the direct sum gives
\begin{equation}
\label{eq:projector-decomposition}
W=W^+\oplus fW\qquad\text{where}\qquad  
W^+=\{w\in W \mid ew=0\}. 
\end{equation}
The decomposition \eqref{eq:projector-decomposition} defines a
canonical linear projection
\[
p_W:W\longrightarrow W^+,
\qquad
p_W|_{W^+}=\operatorname{id},
\quad
p_W|_{fW}=0.
\]
This operator is defined on all of $W$, independently of
whether the rational coefficients in \eqref{eq:p} can be
evaluated on every weight space.

\begin{proposition}
\label{prop:projector-on-modules}
Let $W$ be a finite-dimensional $\mathfrak{sl}_2$-module.
Then
\[
p_WW=W^+,\qquad \ker p_W=fW,\qquad p_W^2=p_W.
\]
On every nonnegative weight space of $W$, the series
\eqref{eq:p} is termwise defined and its action agrees
with $p_W$.
\end{proposition}

\begin{proof}
The first assertions follow from
\eqref{eq:projector-decomposition}. Since $e$ acts locally
nilpotently, only finitely many terms of \eqref{eq:p} can
contribute on a given vector. At a nonnegative weight
$\lambda$, all denominators
$(\lambda+2)\cdots(\lambda+k+1)$ are nonzero, so the
series can be evaluated term by term.

Decompose $W$ into irreducible summands. On each
highest-weight line, the series acts as the identity.
Every other weight vector in that summand belongs to
$fL_\ell$. On a nonnegative weight space the series is
termwise evaluable, and the identity $pf=0$ shows that
it annihilates these descendant vectors. Thus its action
agrees with $p_W$ on every nonnegative weight space.
\end{proof}

For $V=L_n\otimes L_m$, we apply this construction to the
diagonal $\mathfrak{sl}_2$-action and write $\Delta(p)_V$
for the resulting projection $p_V$. When no confusion can
arise, we abbreviate $\Delta(p)_V$ to $\Delta(p)$. Thus
\[
\Delta(p)V=V^+,\qquad
\ker\Delta(p)=\Delta(f)V,\qquad
V=V^+\oplus\Delta(f)V.
\]

\subsection{Construction of singular vectors}

Recall that $n\geq m$ and $N=n+m$.  For $0\leq r\leq m$,
define
\begin{equation}
\label{eq:singdef}
s_{N-2r}
:=
\Delta(p)(f^r1_n\otimes1_m).
\end{equation}
The vector to which $\Delta(p)$ is applied has weight
$N-2r\geq0$.  Hence the denominators in the series
\eqref{eq:p} are nonzero on this weight space, and the
series has only finitely many nonzero terms.

Since $\Delta(e)\Delta(p)=0$ and
$[\Delta(h),\Delta(p)]=0$, we have
\[
\Delta(e)s_{N-2r}=0,\qquad
\Delta(h)s_{N-2r}=(N-2r)s_{N-2r}.
\]
It remains to show that these vectors are nonzero.

\begin{proposition}
\label{prop:singular-basis}
For each $0\leq r\leq m$, the vector $s_{N-2r}$ is
nonzero.  Consequently,
\[
V^+_{N-2r}=\mathbb C s_{N-2r},
\qquad
V^+=\bigoplus_{r=0}^{m}\mathbb C s_{N-2r}.
\]
\end{proposition}

\begin{proof}
Let
\[
u_a:=f^a1_n\otimes f^{r-a}1_m
\qquad(0\leq a\leq r).
\]
These vectors form a basis of $V_{N-2r}$, since
$r\leq m\leq n$.  
Define a linear functional
$\sigma:V_{N-2r}\to\mathbb C$ by
\[
\sigma(u_a)=(-1)^{r-a}.
\]
The space $V_{N-2r+2}$ has basis
$f^a1_n\otimes f^{r-1-a}1_m$ for $0\leq a<r$.
Applying $\Delta(f)$ to such a basis vector gives
$u_{a+1}+u_a$, which is annihilated by $\sigma$.
Thus $\sigma$ vanishes on
$\Delta(f)V_{N-2r+2}$, whereas
$\sigma(f^r1_n\otimes1_m)=\sigma(u_r)=1$.

It follows that $f^r1_n\otimes1_m$ does not belong to
$\Delta(f)V$.  By Proposition~\ref{prop:projector-on-modules},
it is not in the kernel of $\Delta(p)$, so
$s_{N-2r}\ne0$.  The assertions now follow from
\eqref{eq:singular-decomposition}.
\end{proof}

In particular, $s_N=1_n\otimes1_m$.  The first
nontrivial example is
\[
s_{N-2}
=
\frac{m}{N}\,f1_n\otimes1_m
-
\frac{n}{N}\,1_n\otimes f1_m
\qquad (m\geq1).
\]
A closed expression for all $s_{N-2r}$ is given in
Appendix~\ref{app:explicit}.

\subsection{Operators on singular vectors}

For operators acting on the two tensor factors, we use the
notation
\[
e_1=e\otimes1,\quad e_2=1\otimes e,\qquad
f_1=f\otimes1,\quad f_2=1\otimes f,
\]
and similarly $h_1,h_2$.  Thus
\[
\Delta(e)=e_1+e_2,\qquad
\Delta(f)=f_1+f_2,\qquad
\Delta(h)=h_1+h_2.
\]

Define operators on the singular subspace $V^+$ by
\begin{equation}
\label{eq:projected-operators}
z^-=\Delta(p)f_1\big|_{V^+},
\qquad
z^+=\Delta(p)e_1\big|_{V^+}.
\end{equation}
The final projection ensures that both operators preserve $V^+$.
For calculations on the full tensor product $V$, we extend
these operators by zero on $\Delta(f)V$.  Their extensions
are then represented by
\begin{equation}
\label{eq:sandwiched-operators}
z^-=\Delta(p)f_1\Delta(p),
\qquad
z^+=\Delta(p)e_1\Delta(p).
\end{equation}

Because $\Delta(p)$ commutes with $\Delta(h)$, while
\[
[\Delta(h),f_1]=-2f_1,\qquad
[\Delta(h),e_1]=2e_1,
\]
the projected operators satisfy
\begin{equation}
\label{eq:projected-weights}
[\Delta(h),z^-]=-2z^-,
\qquad
[\Delta(h),z^+]=2z^+.
\end{equation}
Consequently,
\[
z^-V^+_\lambda\subseteq V^+_{\lambda-2},
\qquad
z^+V^+_\lambda\subseteq V^+_{\lambda+2}.
\]
In particular, for the singular vectors constructed above,
\[
z^-s_{N-2r}\in\mathbb C s_{N-2r-2}
\quad(0\leq r<m),
\qquad
z^+s_{N-2r}\in\mathbb C s_{N-2r+2}
\quad(1\leq r\leq m).
\]
At the endpoints, $z^-s_{n-m}=0$ and $z^+s_N=0$.

The normalization in \eqref{eq:singdef} makes the action of
$z^-$ particularly simple:
\begin{equation}
\label{eq:zminus-action}
z^-s_{N-2r}=s_{N-2r-2}
\qquad (0\leq r<m).
\end{equation}
Indeed, put $q_r=f_1^r(1_n\otimes1_m)$, so that
$s_{N-2r}=\Delta(p)q_r$.  By
Proposition~\ref{prop:projector-on-modules},
\[
q_r-s_{N-2r}\in\ker\Delta(p)=\Delta(f)V.
\]
Since $f_1$ commutes with $\Delta(f)=f_1+f_2$, it preserves
$\Delta(f)V$.  Applying $\Delta(p)f_1$ therefore gives
\[
z^-s_{N-2r}
=\Delta(p)f_1s_{N-2r}
=\Delta(p)f_1q_r
=\Delta(p)q_{r+1}
=s_{N-2r-2}.
\]
Together with $s_N=1_n\otimes1_m$, this implies
\begin{equation}
\label{eq:singular-chain}
%\label{eq:zminus-action}
s_{N-2r}=(z^-)^r s_N
\qquad (0\leq r\leq m).
\end{equation}

The operator $z^+$ moves in the opposite direction.
Since each singular weight space is one-dimensional, there
are scalars $\alpha_r$ such that
\begin{equation}
\label{eq:zplus-action}
z^+s_{N-2r}=\alpha_r s_{N-2r+2}
\qquad (1\leq r\leq m).
\end{equation}
These scalars are nonzero.  To see this,
equip each $L_n$ with the positive-definite Hermitian
inner product for which the vectors $v_k^{(n)}=f^k1_n$
are orthogonal and
\[
\langle 1_n,1_n\rangle=1,\qquad
\lVert v_k^{(n)}\rVert^2
=\langle v_k^{(n)}, v_k^{(n)}\rangle
=
\frac{k!n!}{(n-k)!}
\quad(0\leq k\leq n).
\]
We take the inner product to be linear in the first
argument.  With this normalization, $e^\dagger=f$ and
$h^\dagger=h$.  Give $V=L_n\otimes L_m$ the tensor-product
inner product, defined by
\[
\langle u_1\otimes u_2,\,
        v_1\otimes v_2\rangle_V
=
\langle u_1,v_1\rangle_{L_n}
\langle u_2,v_2\rangle_{L_m},
\]
and write $\lVert v\rVert^2=\langle v,v\rangle_V$.
Since $\Delta(e)^\dagger=\Delta(f)$, we have
$\Delta(f)V=(\ker\Delta(e))^\perp$.  Thus $\Delta(p)$ is
the orthogonal projection onto $V^+$, and
\eqref{eq:sandwiched-operators} gives
$(z^-)^\dagger=z^+$ on $V^+$.
Note that 
$z^-s_{N-2r+2}=s_{N-2r}$ in \eqref{eq:zminus-action},
and all vectors $s_{N-2r}$ are nonzero by
Proposition~\ref{prop:singular-basis}.  
Taking the inner product of
\eqref{eq:zplus-action} with $s_{N-2r+2}$, we obtain
\[
\alpha_r\lVert s_{N-2r+2}\rVert^2
=\langle z^+s_{N-2r},s_{N-2r+2}\rangle_V
=
\lVert s_{N-2r}\rVert^2>0. 
\]
Hence $\alpha_r>0$ for $1\leq r\leq m$.
In summary, we obtain the following proposition. 

\begin{proposition}
\label{prop:projected-operators-action}
The operators $z^\pm$ act on $V^+$ as follows:
\[
[\Delta(h),z^\pm]=\pm2z^\pm.
\]
For $0\leq r<m$ and $1\leq r\leq m$, respectively,
\[
z^-s_{N-2r}=s_{N-2r-2},
\qquad
z^+s_{N-2r}=\alpha_r s_{N-2r+2},
\]
where $\alpha_r>0$.  At the endpoints,
$z^-s_{n-m}=0$, $z^+s_N=0$. 
In particular,
\[
s_{N-2r}=(z^-)^r s_N
\qquad(0\leq r\leq m).
\]
\end{proposition}

Thus $z^-$ and $z^+$ connect the singular vectors in the chain
\[
s_N
\ \underset{z^+}{\overset{z^-}{\rightleftarrows}}\
s_{N-2}
\ \underset{z^+}{\overset{z^-}{\rightleftarrows}}\
\cdots
\ \underset{z^+}{\overset{z^-}{\rightleftarrows}}\
s_{n-m}.
\]
We now turn to the algebra containing these projected
operators and the relations governing their multiplication.

\section{The diagonal Mickelsson--Zhelobenko algebra}
\label{sec:mz}

In the preceding section, we saw that the projected operators
$z^\pm$ move between the singular vectors of $L_n\otimes L_m$.
We now place these operators in an algebraic framework: the
{\it Mickelsson--Zhelobenko algebra} associated with the diagonal
embedding of $\mathfrak{sl}_2$. In the terminology of
Khoroshkin--Ogievetsky, this is a rank-one instance of
\emph{diagonal reduction} \cite{KO}.

\subsection{The diagonal pair and the Mickelsson--Zhelobenko algebra}
\label{subsec:diagonal-mz}

Let
\[
\mathfrak g=\mathfrak{sl}_2\oplus\mathfrak{sl}_2,
\qquad
\mathfrak k=\mathfrak{sl}_2^{\mathrm{diag}}
=\{(x,x)\mid x\in\mathfrak{sl}_2\}.
\]
We identify $U(\mathfrak g)$ with
$U(\mathfrak{sl}_2)\otimes U(\mathfrak{sl}_2)$,
and $U(\mathfrak k)$ with
$\Delta(U(\mathfrak{sl}_2))$,
the image of the coproduct $\Delta$.
Write
\[
E_{\mathrm d}=\Delta(e)=e_1+e_2,\qquad
F_{\mathrm d}=\Delta(f)=f_1+f_2,\qquad
D=\Delta(h)=h_1+h_2.
\]
The anti-diagonal subspace
\[
\mathfrak p=\{(x,-x)\mid x\in\mathfrak{sl}_2\}
\]
gives the symmetric decomposition
$\mathfrak g=\mathfrak k\oplus\mathfrak p$.

Consider the triangular decomposition
\[
\mathfrak k=\mathfrak k^-\oplus\mathfrak k^0
             \oplus\mathfrak k^+,
\qquad
\mathfrak k^-=\mathbb C F_{\mathrm d},\quad
\mathfrak k^0=\mathbb C D,\quad
\mathfrak k^+=\mathbb C E_{\mathrm d}.
\]
Let $R(\alg{k}^0)=\mathbb C(D)$ be the field of fractions of
$U(\mathfrak k^0)=\mathbb C[D]$, and let $U'(\mathfrak g)$
denote the localization of $U(\mathfrak g)$ with respect to
the nonzero elements of $\mathbb C[D]$, 
\[
U'(\mathfrak g)
=
U(\mathfrak g)\otimes_{U(\alg{k}^0)}R(\alg{k}^0).
\]
Its multiplication
retains the weight shifts: if $[D,x]=\lambda x$, then
\[
q(D)x=xq(D+\lambda)
\qquad\bigl(q(D)\in\mathbb C(D)\bigr).
\]
Set
\[
J=U'(\mathfrak g)E_{\mathrm d},
\qquad
M(\alg{g},\alg{k})=U'(\mathfrak g)/J.
\]

\begin{definition}
\label{def:diagonal-mz}
The {\it Mickelsson--Zhelobenko algebra} ({\it MZ algebra}, in short) 
of the pair
$(\mathfrak g,\mathfrak k)$ is the space of
$\mathfrak k^+$-singular elements in $M(\mathfrak g,\mathfrak k)$:
\[
Z(\mathfrak g,\mathfrak k)
=
\{\,\bar x\in M(\alg{g},\alg{k})\mid E_{\mathrm d}\bar x=0\,\}.
\]
Its multiplication is induced by the identification
\[
Z(\mathfrak g,\mathfrak k)
\simeq \operatorname{Norm}(J)/J,
\qquad
\operatorname{Norm}(J)
=
\{x\in U'(\mathfrak g)\mid Jx\subseteq J\}.
\]
\end{definition}

Indeed, since $J$ is generated as a left ideal by
$E_{\mathrm d}$, the condition $Jx\subseteq J$ is equivalent
to $E_{\mathrm d}x\in J$. Thus the two descriptions in
Definition~\ref{def:diagonal-mz} agree. The quotient 
$M(\mathfrak g,\mathfrak k)$ is
a left $U'(\mathfrak g)$-module, while multiplication of
representatives in $\operatorname{Norm}(J)$ gives the algebra
structure on $Z(\mathfrak g,\mathfrak k)$.

\subsection{Projected generators and their action}
\label{subsec:mz-generators}

The operators $z^\pm$ were defined in
Section~\ref{sec:projector} on the singular subspace of
$L_n\otimes L_m$. We now regard them as the actions of
elements of $Z(\mathfrak g,\mathfrak k)$, denoted by the
same symbols. Their classes in $M(\alg{g},\alg{k})$ are represented by
\[
z^+=\Delta(p)e_1\Delta(p),\qquad
z^-=\Delta(p)f_1\Delta(p).
\]
Throughout, we suppress the notation for classes modulo
$J$. Since $\Delta(p)=1$ modulo $J$, the rightmost projector
does not change these classes; we retain it to make products
of projected elements explicit.

These expressions involve only finitely many terms modulo
$J$. Indeed,
\[
E_{\mathrm d}e_1\in J,\qquad
E_{\mathrm d}^{\,3}f_1\in J.
\]
The projector identity $E_{\mathrm d}\Delta(p)=0$ then
shows that $z^\pm$ are $E_{\mathrm d}$-singular in $M(\alg{g},\alg{k})$.
Thus $z^\pm\in Z(\mathfrak g,\mathfrak k)$.
Explicit representatives of $z^\pm$ modulo $J$ are given
in Appendix~\ref{app:MZgen}.

Since $[D,\Delta(p)]=0$, while
$[D,e_1]=2e_1$ and $[D,f_1]=-2f_1$, we have
\begin{equation}
\label{eq:mz-weight-relations}
[D,z^+]=2z^+,\qquad [D,z^-]=-2z^-.
\end{equation}

On $V=L_n\otimes L_m$, let $\Delta(p)$ denote the projection
onto $V^+$ along $F_{\mathrm d}V$, as in
Section~\ref{sec:projector}. For $v\in V^+$ we then have
\[
z^+v=\Delta(p)e_1v,\qquad
z^-v=\Delta(p)f_1v.
\]
Consequently, the elements $z^\pm$ recover the operators
that move between the singular vectors in
Section~\ref{sec:projector}.

\subsection{The quadratic relation}
\label{subsec:mz-relation}

Let
\[
T=h^2+2h+4fe
\]
be the quadratic Casimir element of $U(\mathfrak{sl}_2)$, and set
\[
T_1=T\otimes1,\qquad T_2=1\otimes T.
\]
Both $T_1$ and $T_2$ are central in $Z(\mathfrak g,\mathfrak k)$.
The coproduct is calculated as 
\[
\De(T)=T_1+T_2+2T_{12}, \qquad \text{where}
\qquad 
T_{12}=h_1h_2+2(f_1e_2+f_2e_1). 
\]
Sandwiching it by $\De(p)$, we have 
\begin{align}
\label{eq:h1h2}
\De(p)D(D+2)\De(p)=\De(p)
\left[T_1+T_2+2h_1h_2-8f_1e_1\right]
\De(p). 
\end{align}

\begin{proposition}
\label{prop:mz-relations}
The elements $D,z^+,z^-\in Z(\mathfrak g,\mathfrak k)$
satisfy
\[
[D,z^+]=2z^+,\qquad [D,z^-]=-2z^-,
\]
and
\begin{equation}
\label{eq:mzrat}
z^+z^-
=
z^-z^+\frac{(D+2)(D+3)}{D(D+1)}
+
\frac{D(D+2)-T_1-T_2}{2D}.
\end{equation}
\end{proposition}

\begin{proof}
The relations involving $D$ were established in
\eqref{eq:mz-weight-relations}. We prove the quadratic
relation. In this proof write $p=\Delta(p)$, and keep
all rational functions of $D$ on the right.

Since $E_{\mathrm d}e_1p=0$, we have $pe_1p=e_1p$.
Consequently,
\begin{equation}
\label{eq:mz-minus-plus}
z^-z^+=pf_1pe_1p=pf_1e_1p.
\end{equation}
To calculate the product in the opposite order, note that
\[
E_{\mathrm d}f_1p=h_1p,\qquad
E_{\mathrm d}^{\,2}f_1p=-2e_1p,\qquad
E_{\mathrm d}^{\,3}f_1p=0.
\]
Thus only the first three terms of the extremal projector
contribute to $z^+z^-$, that is,  
\begin{align}
z^+z^-&=pe_1pf_1p
=pe_1\left[1-F_{\rm d}E_{\rm d}\frac{1}{D+2}
+F_{\rm d}^2E_{\rm d}^2\frac{1}{2(D+2)(D+3)}\right]f_1p
\notag \\
&=pe_1f_1p
-pe_1F_{\rm d}E_{\rm d}f_1p \frac{1}{D}
+pe_1F_{\rm d}^2E_{\rm d}^2f_1p\frac{1}{2D(D+1)}. 
\end{align}
By the properties $E_{\rm d}p=pF_{\rm d}=0$, 
\begin{align}
z^+z^-
&=pe_1f_1p
-p[e_1,F_{\rm d}][E_{\rm d},f_1]p \frac{1}{D}
+p[e_1,F_{\rm d}^2][E_{\rm d}^2,f_1]p\frac{1}{2D(D+1)}
\notag \\
&=pf_1e_1p+ph_1p-p\frac{h_1^2}{D}p+2pf_1e_1p\frac{1}{D(D+1)}
\notag \\
&=pf_1e_1p\left[1+\frac{2}{D(D+1)}\right]+p\frac{h_1h_2}{D}p. 
\end{align}
Together with \eqref{eq:mz-minus-plus}, we have 
\begin{align}
z^+z^-
&=z^-z^+ \left[1+\frac{2}{D(D+1)}\right]+p\frac{h_1h_2}{D}p. 
\end{align}
The relation \eqref{eq:h1h2} gives  
\[
ph_1h_2p=\frac{1}{2} p\left[D(D+2)-T_1-T_2 \right]p
+4z^-z^+. 
\]
Substituting this for the above expression, 
we obtain
\begin{align}
z^+z^-
&=z^-z^+ \left[1+\frac{2}{D(D+1)}+\frac{4}{D}\right]+
\frac{1}{2D} \left[D(D+2)-T_1-T_2 \right]p. 
\notag\\
&=z^-z^+ \frac{(D+2)(D+3)}{D(D+1)}+
\frac{1}{2D} \left[D(D+2)-T_1-T_2 \right]p. 
\end{align}
Since the rightmost $p$ represents the identity in 
$Z(\mathfrak g,\mathfrak k)$, this proves 
\eqref{eq:mzrat}.
\end{proof}

The relation \eqref{eq:mzrat}
is the rank-one diagonal reduction relation, in the normalization
used here; compare the explicit $DR(\mathfrak{sl}_2)$ relations of
Khoroshkin--Ogievetsky \cite{KO}.

By evaluating the diagonal Casimir operator on $\Delta(p)$, 
\[
T_{12}\De(p)=\frac12(D(D+2)-T_1-T_2)\De(p), 
\]
the relation is also expressed  as 
\begin{equation}
z^+z^-
=
z^-z^+\frac{(D+2)(D+3)}{D(D+1)}
+
\frac{T_{12}}{D}.
\end{equation}

\section{From the MZ algebra to a cubic Smith algebra}
\label{sec:smith}

The quadratic relation of Section~\ref{sec:mz} contains
a rational coefficient multiplying $z^-z^+$. We seek a
rescaling of $z^\pm$ that absorbs this coefficient and
turns the relation into an ordinary commutation relation.
We first give a general sufficient condition for such a
rescaling, then apply it to the diagonal MZ algebra.
In this case the resulting commutator is cubic and
defines a Smith algebra.

\subsection{A rescaling criterion for commutator relations}
\label{subsec:linearization}

We begin with a sufficient condition for removing the
rational product term from a quadratic relation in the MZ algebra.

\begin{lemma}
\label{lem:linearization}
Suppose that an associative $\mathbb C$-algebra contains
the rational function field $\mathbb C(D)$ and elements
$D, z^+,z^-$ satisfying 
\[
[D,z^+]=2z^+,
\qquad
[D,z^-]=-2z^-,
\qquad
z^+z^-=z^-z^+Q(D)+R(D),
\]
where $Q(D)$ and $R(D)$ commute with $D$.
Assume that a nonzero rational function $b(D)$ satisfies
\begin{equation}
\label{eq:linearization-condition}
Q(D)=\frac{b(D+2)}{b(D)}.
\end{equation}
For any nonzero rational function $a(D)$, set
\begin{equation}
\label{eq:general-rescaling}
E=a(D)z^+,\qquad
F=z^-\frac{b(D)}{a(D)}.
\end{equation}
Then
\[
[D,E]=2E,\qquad [D,F]=-2F,
\qquad
[E,F]=b(D)R(D).
\]
\end{lemma}

\begin{proof}
The relations with $D$ follow from those for $z^\pm$.
Multiplying the quadratic relation on the right by $b(D)$, 
we obtain
\begin{align}
z^+z^-b(D)=z^-z^+b(D+2)+b(D)R(D). 
\notag 
\end{align}
Using the weight relation for $z^+$ and noting 
$z^+z^-$ has weight zero, it becomes  
\begin{align}
a(D)z^+\cdot z^-\frac{b(D)}{a(D)}=z^-\frac{b(D)}{a(D)}\cdot a(D)z^++b(D)R(D).
\notag 
\end{align}
This is the desired relation
\[
EF=FE+b(D)R(D).
\]
\end{proof}

The function $b(D)$ removes the coefficient of $z^-z^+$
from the commutator. The choice of $a(D)$ does not affect
the resulting relation; it records a freedom in the
normalization of $E$ and $F$.

\subsection{A realization of the cubic Smith algebra}
\label{subsec:cubic-smith}

We now apply Lemma~\ref{lem:linearization} to the
quadratic relation \eqref{eq:mzrat}. In its notation,
\[
Q(D)=\frac{(D+2)(D+3)}{D(D+1)},
\qquad
R(D)=\frac{D(D+2)-T_1-T_2}{2D}.
\]
The coefficient $Q(D)$ has the required form:
\[
Q(D)=\frac{b(D+2)}{b(D)}
\qquad\text{for}\qquad
b(D)=4D(D+1).
\]
The constant factor $4$ does not affect this ratio.
It is chosen so that the resulting commutator has the
normalization
\[
b(D)R(D)
=
2(D+1)\bigl(D(D+2)-T_1-T_2\bigr).
\]
Thus the diagonal MZ relation satisfies the rescaling
criterion. It remains to choose the gauge function $a(D)$
and identify the resulting algebra.

We fix the gauge $a(D)=1$ and define
\begin{equation}
\label{eq:smith-generators}
E:=z^+,\qquad
F:=4z^-D(D+1),\qquad
H:=D+1.
\end{equation}
Set
\[
c:=T_1+T_2+1.
\]
Since $T_1$ and $T_2$ are central in
$Z(\mathfrak g,\mathfrak k)$, so is $c$.

Smith introduced a class of algebras with generators
$E,F,H$ satisfying $[H,E]=2E$, $[H,F]=-2F$, and
$[E,F]=P(H)$ for a polynomial $P\in \mathbb{C}[H]$ in \cite{S}. 
When $P$ is cubic, 
relations of this form are also referred to
as {\it Higgs-algebra} relations in the physics literature
\cite{Hig,Zhed}. We use the term \emph{cubic Smith
relations} for the normalization obtained below.

\begin{theorem}
\label{thm:cubic}
The elements $E,F,H$ of
$Z(\mathfrak g,\mathfrak k)$ 
given in \eqref{eq:smith-generators}
satisfy
\begin{equation}
\label{eq:cubic}
[H,E]=2E,\qquad
[H,F]=-2F,\qquad
[E,F]=2(H^3-cH).
\end{equation}
Thus they define a realization of the cubic Smith
relations inside the diagonal MZ algebra.
\end{theorem}

\begin{proof}
The first two relations follow from
\eqref{eq:mz-weight-relations}. Lemma~\ref{lem:linearization}
gives
\[
[E,F]
=b(D)R(D)
=2(D+1)\bigl(D(D+2)-T_1-T_2\bigr).
\]
Since $H=D+1$ and $D(D+2)=H^2-1$, this becomes
\[
[E,F]=2H\bigl(H^2-(T_1+T_2+1)\bigr)
     =2(H^3-cH).
\]
This completes the proof. 
\end{proof}

\begin{remark}
The choice $a(D)=1$ fixes a convenient normalization.
More generally, replacing $E$ and $F$ by
\[
E_a=a(D)z^+,\qquad
F_a=z^-\frac{4D(D+1)}{a(D)}
\]
for a nonzero rational function $a(D)$ leaves $E_aF_a$
and $F_aE_a$ unchanged. Hence the Smith relations and
the Smith Casimir $\Omega$, which we discuss later, are independent of this choice.
The function $a(D)$ only redistributes the coefficients
between the raising and lowering operators on a weight
module.
\end{remark}

Higgs algebras have also appeared in the setting of
{\it transvector algebras} associated with harmonic analysis
\cite{DER,Eel}. In particular, Eelbode identifies Higgs
algebras within a transvector algebra related to the
harmonic Fischer decomposition \cite{Eel}. The reduction
pair and representation considered here are different:
our construction uses the diagonal embedding
$\mathfrak{sl}_2\subset\mathfrak{sl}_2\oplus\mathfrak{sl}_2$
and acts on tensor-product singular vectors.

\subsection{The Smith Casimir and central parameters}
\label{subsec:smith-casimir}

We first recall the abstract construction of the Smith algebra
and its Casimir element \cite{S}. 
For a polynomial $P\in \mathbb C[K]$, let
$\mathcal S(P)$ be the algebra generated by $X,Y,K$
subject to
\[
[K,X]=2X,\qquad [K,Y]=-2Y,\qquad [X,Y]=P(K).
\]
Choose a polynomial $\Phi\in \mathbb C[K]$  satisfying
\begin{equation}
\label{eq:smith-potential-difference}
\Phi(K)-\Phi(K-2)=P(K), 
\end{equation}
which is  called the {\it Smith potential}. 
Then
\begin{equation}
\label{eq:abstract-smith-casimir}
\Omega_{\mathcal S}
:=YX+\Phi(K)
 =XY+\Phi(K-2)
\end{equation}
is central in $\mathcal S(P)$ and defines 
the Casimir element in  $\mathcal S(P)$. 
The equality follows
from $XY-YX=P(K)$, and the difference equation shows
that $\Omega_{\mathcal S}$ commutes with $X$ and $Y$.
For instance, 
\begin{align}
X\Omega_{\mathcal S}
&=X(YX+\Phi(K))=XYX+\Phi(K-2)X
=(XY+\Phi(K-2))X=\Omega_{\mathcal S}X. 
\notag 
\end{align}
The potential, and hence this choice of Casimir, 
is determined up to an additive constant.

In the cubic case considered here, the polynomial is
$P(K)=2(K^3-cK)$, with $c$ a central parameter. We choose
\begin{equation}
\label{eq:potential}
\Phi(K)
=
\frac14K(K+2)\bigl(K(K+2)-2c\bigr).
\end{equation}
Theorem~\ref{thm:cubic} gives a homomorphism from this
cubic Smith algebra to the MZ algebra 
$Z(\mathfrak g,\mathfrak k)$ by
\[
X\longmapsto E,\qquad
Y\longmapsto F,\qquad
K\longmapsto H,\qquad
c\longmapsto T_1+T_2+1.
\]
We denote the image of $\Omega_{\mathcal S}$ by
$\Omega$. Thus, in our MZ realization,
\[
\Omega=FE+\Phi(H)=EF+\Phi(H-2).
\]

From now on, we use $c$ also for its image
$T_1+T_2+1$ in $Z(\mathfrak g,\mathfrak k)$.

\begin{proposition}
\label{prop:smith-central-parameters}
In the diagonal MZ realization, the Smith Casimir
$\Omega=FE+\Phi(H)$ is given by
\begin{equation}
\label{eq:centralmap}
c=T_1+T_2+1,\qquad
\Omega=
\frac14\bigl(-(T_1-T_2)^2
             +2(T_1+T_2)+3\bigr).
\end{equation}
Thus both Smith central parameters are expressed in
terms of the Casimir elements of the two tensor factors.
\end{proposition}

\begin{proof}
The formula $c=T_1+T_2+1$ was established in
Theorem~\ref{thm:cubic}. We identify $\Omega$.
We use the following abbreviations:
\[
p=\Delta(p),\qquad
S=T_1+T_2,\qquad
%Z=z^-z^+,\qquad
A=T_1-T_2.
\]
All calculations below take place in the diagonal MZ
algebra. In particular, $p$ represents its identity.

First, \eqref{eq:mz-minus-plus} and
\eqref{eq:smith-generators} give
\begin{equation}
\label{eq:omega-first}
z^-z^+=pf_1e_1p,\qquad
FE=4z^-z^+(D+2)(D+3),\qquad
\Omega=FE+\Phi(D+1).
\end{equation}
We shall express $z^-z^+$ in terms of $D,T_1,T_2$.
Since $E_{\mathrm d}p=0$ and $pF_{\mathrm d}=0$,
\[
pf_2e_2p
=p(F_{\mathrm d}-f_1)(E_{\mathrm d}-e_1)p
=pf_1e_1p=z^-z^+.
\]
Taking the difference of the two factor Casimirs,
and using $h_2=D-h_1$, we obtain
\[
A
=p\bigl(h_1^2-h_2^2+2h_1-2h_2\bigr)p
=(D+2)(2ph_1p-D).
\]
Thus, introducing the projected $h_1$ generator 
by $z^0=ph_1p$, we have 
\begin{equation}
\label{eq:projected-h1}
z^0=\frac{1}{2}\left[D+\frac{A}{D+2}\right].
\end{equation}

Next we compare $ph_1^2p$ with $(z^0)^2$.
By the relation 
\begin{align}
ph_1^2p
&=ph_1(p+(1-p))h_1p 
=(z^0)^2+ph_1(1-p)h_1p,  
\end{align}
it is sufficient to evaluate $(1-p)h_1p$. 
The identities
\[
E_{\mathrm d}h_1p=-2e_1p,\qquad
E_{\mathrm d}^{\,2}h_1p=0
\]
show that the extremal-projector series terminates
after its linear term on $h_1p$. Hence
\[
(1-p)h_1p
=-2F_{\mathrm d}e_1\frac1{D+2}p.
\]
Using $pF_{\mathrm d}=0$ and
$[h_1,F_{\mathrm d}]=-2f_1$, we find
\begin{align}
ph_1^2p
%&=ph_1ph_1p+ph_1(1-p)h_1p \notag\\
&=(z^0)^2
  -2ph_1F_{\mathrm d}e_1\frac1{D+2}p
 =(z^0)^2+\frac4{D+2}z^-z^+.
\label{eq:projected-h1-square}
\end{align}

The first factor Casimir now gives
\[
T_1=ph_1^2p+2z^0+4pf_1e_1p
=(z^0)^2+2z^0+\frac{4(D+3)}{D+2}z^-z^+.
\]
Solving for $z^-z^+$ yields
\begin{equation}
\label{eq:Z-in-casimirs}
z^-z^+=\frac{D+2}{4(D+3)}
   \bigl(T_1-z^0(z^0+2)\bigr).
\end{equation}
Substituting \eqref{eq:projected-h1} into this 
expression and setting $u=D+2$, we obtain
\begin{align}
FE&=4z^-z^+(D+2)(D+3)
\notag \\
&=(D+2)^2\bigl(T_1-z^0(z^0+2)\bigr)
\notag \\
&=\frac{1}{4}\left(-u^4+2(S+2)u^2-A^2\right).  
\label{eq:FE}
\end{align}

On the other hand, 
since $c=S+1$, 
our choice of the Smith potential $\Phi(H)$ 
in \eqref{eq:potential} with $H=D+1$ gives
\begin{align}
\Phi(D+1)&=\frac14(u^2-1)(u^2-2S-3)
\notag \\
&=\frac14\left(u^4-2(S+2)u^2+2S+3\right).
\label{eq:Phi}
\end{align}

Adding the potential \eqref{eq:Phi} to the 
expression for $FE$ in \eqref{eq:FE},
all terms involving $u$ cancel as expected. 
Therefore, the Smith Casimir is given by
\begin{align}
\Omega&=FE+\Phi(D+1)
\notag \\
&=\frac14\bigl(-A^2+2S+3\bigr)
=\frac14\bigl(-(T_1-T_2)^2+2(T_1+T_2)+3\bigr),
\notag
\end{align}
as claimed.
\end{proof}

\begin{remark}
\label{rem:potential-ambiguity}
The difference equation
\eqref{eq:smith-potential-difference} determines
$\Phi$ up to an additive element $q$ of the central
coefficient ring. Replacing $\Phi$ by $\Phi+q$ leaves
the Smith relations unchanged, but replaces the
associated Casimir by
\[
\Omega=FE+\Phi(H)
\quad\longmapsto\quad
\Omega+q.
\]
Thus the explicit formula for $\Omega$ in
Proposition~\ref{prop:smith-central-parameters}
corresponds to the particular choice of potential
in \eqref{eq:potential}.
\end{remark}

\begin{remark}
The pair $(T_1,T_2)$ therefore separates naturally into 
two pieces of Smith-algebra data $(c,\Omega)$: 
their sum determines the deformation parameter $c$, 
while their difference enters the Smith Casimir $\Omega$.
\end{remark}

\section{The Smith module arising from $\alg{sl}_2$ 
tensor products}
\label{sec:representation}

The preceding sections constructed cubic Smith generators
inside the diagonal MZ algebra.
We now study their action on the singular subspace of
$L_n\otimes L_m$.
Here the projected operators are understood using the
globally defined projection of Section~\ref{sec:projector}.
We consider the action of the Smith generators,
without requiring an action of the full localized
MZ algebra on this subspace.

In Section~\ref{sec:smith-module-action}, we describe their
action on the singular-vector basis and give small matrix
examples. In Section~\ref{sec:smith-shapovalov}, we compute
the Shapovalov form for cubic Smith highest-weight modules.
Its degeneracy and simplicity conditions allow us to
identify which finite-dimensional simple modules arise
from tensor products of finite-dimensional
$\mathfrak{sl}_2$-modules.

\subsection{The Smith module on singular vectors}
\label{sec:smith-module-action}

Let $n,m\in\mathbb Z_{\geq 0}$, and put
\[
N=n+m,\qquad q=\min\{n,m\}.
\]
For $0\leq r\leq q$, use the singular vectors
\[
s_{N-2r}
=\Delta(p)(f^r1_n\otimes 1_m)
=(z^-)^r(1_n\otimes 1_m).
\]

\begin{theorem}
\label{thm:smith-tensor-module}
The singular subspace $V^+=(L_n\otimes L_m)^+$ is an
irreducible Smith module of dimension $q+1$, with basis
\[
s_N,~s_{N-2},\cdots,~s_{N-2q}.
\]
The generators $E,F,H$ 
in \eqref{eq:smith-generators}
act by
\begin{align}
E s_{N-2r}
&=\frac{r(n-r+1)(m-r+1)(N-r+2)}{(N-2r+2)(N-2r+3)}
\,s_{N-2r+2}
&& (1\leq r\leq q), 
\notag \\
F s_{N-2r}
&=4(N-2r)(N-2r+1)s_{N-2r-2}
&& (0\leq r\leq q-1),
\notag \\
H s_{N-2r}&=(N-2r+1)s_{N-2r}
&& (0\leq r\leq q). 
\notag 
\end{align}
At the endpoints, $Es_N=0$ and $Fs_{|n-m|}=0$.
In particular, the highest weight with respect to $H$
is $N+1$.
The central elements $c$ and $\Omega$ respectively 
act on $V^+$ by the scalars
\begin{align}
c_{n,m}&=n(n+2)+m(m+2)+1, \notag \\
\omega_{n,m}
%&=-\frac14(N+1)(N+3)\bigl((n-m)^2-1\bigr).  
&=-\frac14(N+1)(N+3)(n-m+1)(n-m-1). 
\label{eq:smith-omega-nm}
\end{align}
\end{theorem}

\begin{proof}
The construction in Section~\ref{sec:projector} gives
nonzero singular vectors $s_{N-2r}$ for $0\leq r\leq q$.
Their weights are distinct, and the Clebsch--Gordan rule
shows that they form a basis of $V^+$.
The same argument applies when $n<m$, with
$q=\min(n,m)$.

Since $H=D+1$, its action follows from the weight of
$s_{N-2r}$.
The formula for $F=4z^-D(D+1)$ follows from
$z^-s_{N-2r}=s_{N-2r-2}$ for $r<q$.
The endpoint relations $Es_N=0$ and $Fs_{|n-m|}=0$
follow from the corresponding relations for $z^\pm$.
Moreover, since $c=T_1+T_2+1$, the central element $c$
acts on $V^+$ by the scalar
\[
c_{n,m}=n(n+2)+m(m+2)+1.
\]

We next determine the action of $E$.
Write
\[
Es_{N-2r}=\alpha_r s_{N-2r+2}
\qquad (1\leq r\leq q),
\]
and set $\alpha_0=0$, in accordance with $Es_N=0$.
For $1\leq r\leq q$, applying $E$ to
\[
Fs_{N-2r+2}
=4(N-2r+2)(N-2r+3)s_{N-2r}
\]
gives
\[
EFs_{N-2r+2}
=4(N-2r+2)(N-2r+3)\alpha_r s_{N-2r+2}.
\]
On the other hand, the cubic Smith relation
\eqref{eq:cubic} gives
\[
\begin{aligned}
EFs_{N-2r+2}
&=\bigl([E,F]+FE\bigr)s_{N-2r+2}\\
&=\Bigl(
  2(H^3-cH)
  +4(N-2r+4)(N-2r+5)\alpha_{r-1}
  \Bigr)s_{N-2r+2}\\
&=\Bigl(
  \Phi(H)-\Phi(H-2)
  +4(N-2r+4)(N-2r+5)\alpha_{r-1}
  \Bigr)s_{N-2r+2},
\end{aligned}
\]
where $\Phi$ is the Smith potential
\eqref{eq:potential}, evaluated at $c=c_{n,m}$.
Set
\[
\beta_r=4(N-2r+2)(N-2r+3)\alpha_r,
\qquad
\Phi_r=\Phi(N-2r+1).
\]
Comparing the two expressions and using the
$H$-weight of $s_{N-2r+2}$, we obtain
\[
\beta_r=\Phi_{r-1}-\Phi_r+\beta_{r-1}.
\]
Thus $\beta_r+\Phi_r$ is independent of $r$.
Since $\beta_0=0$, it follows that
\[
\alpha_r
=\frac{\Phi(N+1)-\Phi(N-2r+1)}
{4(N-2r+2)(N-2r+3)}.
\]
To factor the numerator, set
$u=N+2$ and $v=N-2r+2$.
By \eqref{eq:potential},
\[
\begin{aligned}
\Phi(N+1)-\Phi(N-2r+1)
&=\frac14(u^2-v^2)
  \bigl(u^2+v^2-2(c_{n,m}+1)\bigr)\\
&=4r(N-r+2)(n-r+1)(m-r+1).
\end{aligned}
\]
This proves the stated formula for the action of $E$.

It remains to determine the scalar $\omega_{n,m}$
by which $\Omega$ acts.
Proposition~\ref{prop:smith-central-parameters} gives
\[
4\omega_{n,m}
=-\bigl(n(n+2)-m(m+2)\bigr)^2
 +2n(n+2)+2m(m+2)+3.
\]
Setting $x=n+1$ and $y=m+1$, we obtain
\[
\begin{aligned}
4\omega_{n,m}
&=-(x^2-y^2)^2+2(x^2+y^2)-1\\
&=-\bigl((x+y)^2-1\bigr)
    \bigl((x-y)^2-1\bigr)\\
&=-(x+y+1)(x+y-1)(x-y+1)(x-y-1).
\end{aligned}
\]
Since $x+y=N+2$ and $x-y=n-m$, this gives
\eqref{eq:smith-omega-nm}.

To prove irreducibility, let $W\subset V^+$ be a nonzero
submodule and choose $0\neq w\in W$.
Since $E$ is nilpotent on $V^+$, there exists 
$l\in \ZZ_{\geq0}$
such that $E^lw\neq0$ and $E^{l+1}w=0$.
The formula for $E$, whose coefficients are nonzero
for $1\leq r\leq q$, shows that
$\ker(E|_{V^+})=\mathbf C s_N$.
Hence $E^lw$ is a nonzero scalar multiple of $s_N$,
so $s_N\in W$.
Repeated application of $F$ then gives all the basis
vectors $s_{N-2r}$, since its coefficients are nonzero
for $0\leq r<q$.
Thus $W=V^+$, proving irreducibility.
\end{proof}

\begin{remark}
The raising coefficients can also be recovered directly
from the Smith Casimir.
Write
\[
Es_{N-2r}=\alpha_r s_{N-2r+2}
\qquad (1\leq r\leq q).
\]
Using $FE=\Omega-\Phi(H)$ and the formula for $F$,
we obtain
\[
4(N-2r+2)(N-2r+3)\alpha_r
=\omega_{n,m}-\Phi(N-2r+1).
\]
Substituting the central parameters $c=c_{n,m}$ and
$\omega_{n,m}$ gives
\[
\omega_{n,m}-\Phi(N-2r+1)
=4r(n-r+1)(m-r+1)(N-r+2),
\]
which yields the same formula for the action of $E$.
\end{remark}

\begin{example}[The one-dimensional case]
\label{ex:smith-one-dimensional}
For $n\geq 0$, the singular subspace of $L_n\otimes L_0$ is
\[
(L_n\otimes L_0)^+
   =\mathbb C s_n,
\qquad
s_n=1_n\otimes 1_0.
\]
The Smith generators act by
\[
Es_n=Fs_n=0,
\qquad
Hs_n=(n+1)s_n.
\]
The central parameters are
\[
c_{n,0}=(n+1)^2,
\qquad
\omega_{n,0}
=-\frac14(n+1)^2(n+3)(n-1).
\]
In particular, $H^3-c_{n,0}H$ acts by zero, as required by
the Smith relation $[E,F]=2(H^3-cH)$.
For $n=0$, we have $H=1, c=1$ and $\omega_{0,0}=3/4$. 
\end{example}

\begin{example}
\label{ex:L11}
For $V=L_1\otimes L_1$, the singular subspace has the
ordered basis $(s_2,s_0)$, where
\[
s_2=1_1\otimes 1_1,\qquad
s_0=\frac12\bigl(f1_1\otimes 1_1-1_1\otimes f1_1\bigr).
\]
The Smith generators are represented by
\[
E=\begin{pmatrix}0&\frac12\\0&0\end{pmatrix},
\qquad
F=\begin{pmatrix}0&0\\24&0\end{pmatrix},
\qquad 
H=\begin{pmatrix}3&0\\0&1\end{pmatrix}. 
\]
Here $T_1=T_2=3$, so $c=7$ and $\omega=15/4$.
In particular,
\[
[E,F]=\begin{pmatrix}12&0\\0&-12\end{pmatrix}
      =2(H^3-7H).
\]
\end{example}

\begin{example}
\label{ex:L32}
For $V=L_3\otimes L_2$, the singular subspace has the
ordered basis $(s_5,s_3,s_1)$. 
See Figure~\ref{fig:L32}. 
In this basis,
\[
E=
\begin{pmatrix}
0&\frac65&0\\
0&0&\frac53\\
0&0&0
\end{pmatrix},
\qquad
F=
\begin{pmatrix}
0&0&0\\
120&0&0\\
0&48&0
\end{pmatrix}, 
\qquad 
H=
\begin{pmatrix}
6&0&0\\
0&4&0\\
0&0&2
\end{pmatrix}.
\]
Here $T_1=15$ and $T_2=8$, so the central parameters are
$c=24$ and $\omega=0$. Directly,
\[
[E,F]=\operatorname{diag}(144,-64,-80)
      =2(H^3-24H).
\]
\begin{figure}[tb]%[h]
  \centering
  \hspace{-8mm}
  \includegraphics[width=0.5\textwidth,  
  trim=50mm 10mm 50mm 5mm,clip]{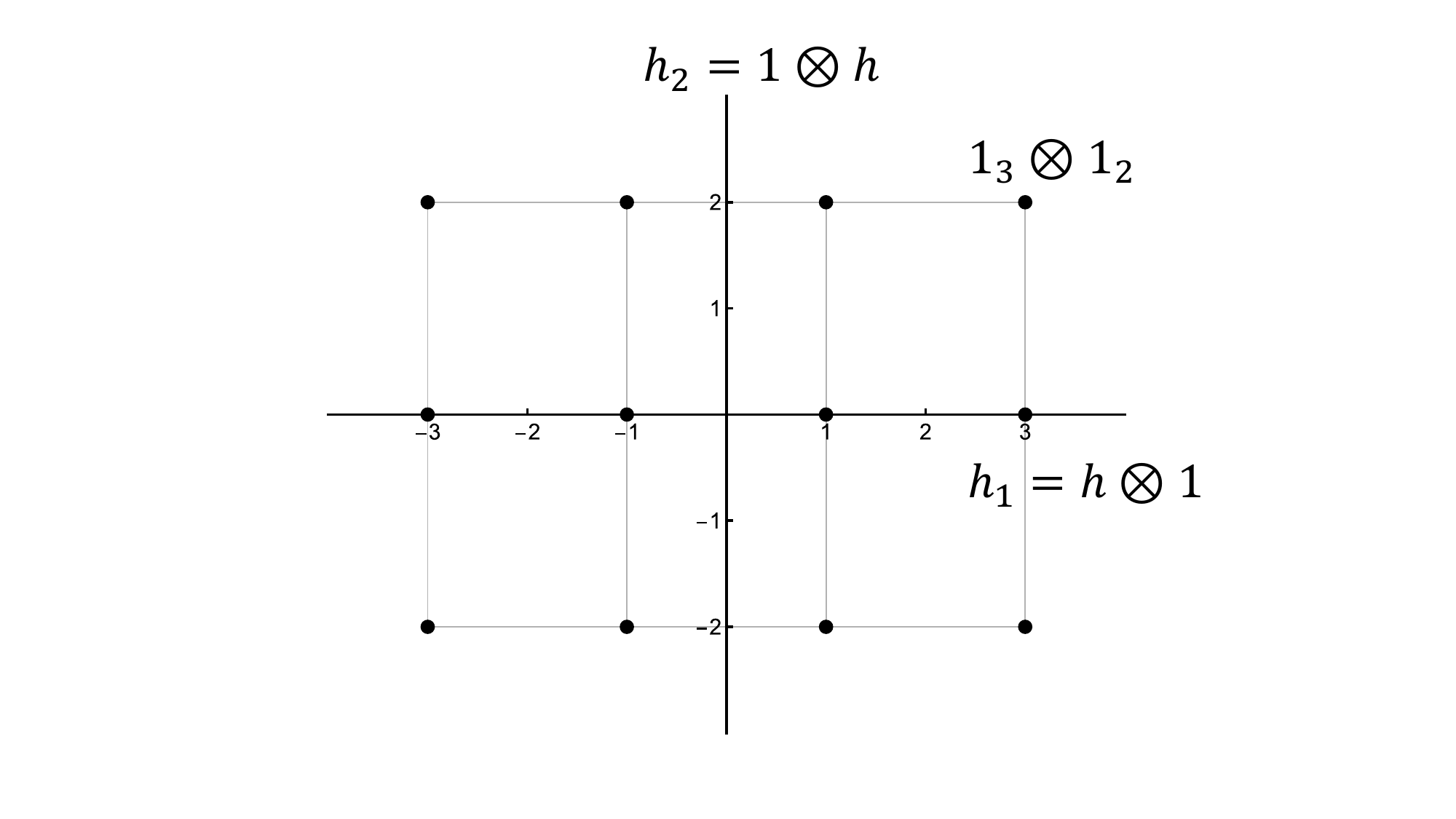}
  \hfill
  \includegraphics[width=0.5\textwidth,  
  trim=50mm 10mm 50mm 5mm,clip]{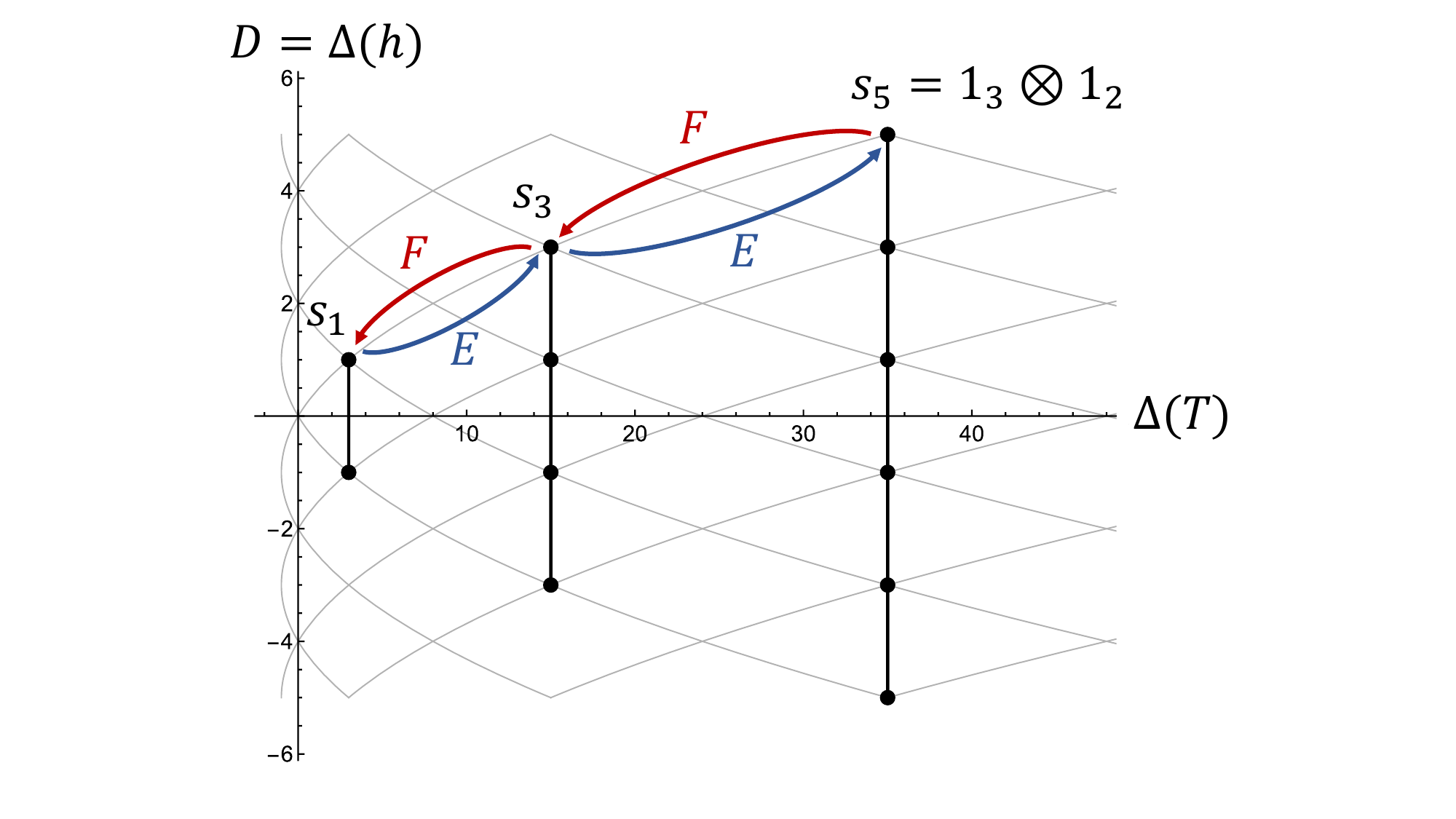}
\caption{The tensor product $L_3\otimes L_2$ in two bases
(Example~\ref{ex:L32}).
Left: the rectangular weight lattice for $(h_1,h_2)$.
Right: the joint spectrum of $(\Delta(T),\Delta(h))$
corresponding to $L_1\oplus L_3\oplus L_5$.
The points representing the highest-weight vectors
$s_1,s_3,s_5$ lie on the parabola $x=y(y+2)$:
each $s_\ell$ has eigenvalues $\ell(\ell+2)$ and $\ell$
under $\Delta(T)$ and $\Delta(h)$, respectively.
The Smith generators $E$ and $F$ connect adjacent
highest-weight lines.}
  \label{fig:L32}
\end{figure}
\end{example}

\begin{example}
\label{ex:smith-two-dimensional-family}
Let $V=L_n\otimes L_1$ with $n\geq1$. The singular
subspace $V^+$ has the ordered basis
\[
(s_{n+1},s_{n-1}),
\qquad
s_{n-1}
=\frac{1}{n+1}f1_n\otimes1_1
 -\frac{n}{n+1}1_n\otimes f1_1.
\]
In this basis, the Smith generators are
\[
E=
\begin{pmatrix}
0&\dfrac{n}{n+1}\\
0&0
\end{pmatrix},
\qquad
F=
\begin{pmatrix}
0&0\\
4(n+1)(n+2)&0
\end{pmatrix}, 
\qquad 
H=
\begin{pmatrix}
n+2&0\\
0&n
\end{pmatrix}. 
\]
Their central characters are
\[
c_{n,1}=(n+1)^2+3,\qquad
\omega_{n,1}
=-\frac14 (n+2)(n+4)n(n-2).
\]
Thus $L_n\otimes L_1$ gives a family of
two-dimensional irreducible Smith modules.
In particular, $\omega_{2,1}=0$.
For $n=1$, it reduces to Example~\ref{ex:L11}. 
\end{example}

\begin{example}
\label{ex:smith-three-dimensional}
Let $V=L_n\otimes L_2$ with $n\geq 2$. The singular
subspace $V^+$ is three-dimensional. In the ordered basis
$(s_{n+2},s_n,s_{n-2})$,
the Smith generators are represented by
\[
E=
\begin{pmatrix}
0&\dfrac{2n}{n+2}&0\\
0&0&\dfrac{2(n-1)(n+2)}{n(n+1)}\\
0&0&0
\end{pmatrix},
\qquad 
F=
\begin{pmatrix}
0&0&0\\
4(n+2)(n+3)&0&0\\
0&4n(n+1)&0
\end{pmatrix},
\]
and $H={\rm diag}(n+3,n+1,n-1)$. 
The central parameters are
\[
c_{n,2}=(n+1)^2+8,\qquad
\omega_{n,2}
=-\frac14(n+3)(n+5)(n-1)(n-3).
\]
Thus these tensor products give a family of
three-dimensional irreducible Smith modules with varying
central characters. In particular, $\omega_{3,2}=0$
as observed in Example \ref{ex:L32}. 
More generally, \eqref{eq:smith-omega-nm} shows that
$\Omega$ acts by zero on $V^+$ whenever $|n-m|=1$;
this does not assert that $\Omega$ vanishes in the MZ algebra.
\end{example}

\subsection{Finite-dimensional Smith modules and the tensor-product range}
\label{sec:smith-shapovalov}

We now compare the modules of Theorem~\ref{thm:smith-tensor-module}
with finite-dimensional simple modules of the abstract cubic
Smith algebra. 
Smith classifies finite-dimensional simple modules as simple
quotients of highest-weight modules
\cite[Lemmas~2.1--2.3]{S}.
In our normalization, the condition for such a quotient
to have a given dimension can be expressed directly
in terms of the Smith potential.
We first derive this condition using a contravariant
bilinear form.

Fix $c\in\mathbb C$ and a highest weight $\lambda\in\mathbb C$.
Let $M(\lambda,c)$ be the highest-weight module for the
abstract generators $X,Y,K$ of Section~\ref{sec:smith},
generated by a vector $v_0$ satisfying
\[
Xv_0=0,\qquad Kv_0=\lambda v_0.
\]
Here the central parameter acts by the chosen scalar $c$.
Set
\[
v_r=Y^rv_0\qquad (r\geq0).
\]
Then $M(\lambda,c)$ has basis $\{v_r\}_{r\geq0}$, with
$Kv_r=(\lambda-2r)v_r$.

The algebra has an anti-involution $\dagger$ given by
\[
X^\dagger=Y,\qquad Y^\dagger=X,\qquad
K^\dagger=K,\qquad c^\dagger=c.
\]
We use the symmetric $\mathbb C$-bilinear form
$(\, ,\,)$ on $M(\lambda,c)$ characterized by
\[
(v_0,v_0)=1,\qquad
(av,w)=(v,a^\dagger w)
\]
for every algebra element $a$ and all $v,w\in M(\lambda,c)$.
This is the Shapovalov form. It is bilinear, with no
complex conjugation, and need not be positive definite.

\begin{lemma}
\label{lem:smith-shapovalov}
Write
\[
\beta_r=\Phi(\lambda)-\Phi(\lambda-2r)
\qquad (r\geq1),
\]
where $\Phi$ is the potential \eqref{eq:potential} with
the chosen value of $c$. Then
\[
Xv_r=\beta_r v_{r-1}.
\]
The basis vectors of distinct weights are orthogonal for
the Shapovalov form, and
\begin{equation}
\label{eq:smith-shapovalov-product}
(v_r,v_r)=\prod_{j=1}^{r}\beta_j
=\prod_{j=1}^{r}
 \bigl[\Phi(\lambda)-\Phi(\lambda-2j)\bigr].
\end{equation}
Moreover, on setting $\mu_r=\lambda-r+1$, one has
\begin{equation}
\label{eq:smith-shapovalov-factor}
\beta_r
=2r \mu_r\bigl(\mu_r^2+r^2-c-1\bigr).
\end{equation}
\end{lemma}

\begin{proof}
The central element
$\Omega_S=YX+\Phi(K)$ acts on $v_0$ by
$\Phi(\lambda)$, and hence acts by this scalar on the
whole highest-weight module. Using the other expression
$\Omega_S=XY+\Phi(K-2)$, we obtain
\[
Xv_r=XYv_{r-1}
=\bigl[\Phi(\lambda)
       -\Phi(\lambda-2r)\bigr]v_{r-1}.
\]
Since $K^\dagger=K$, vectors of distinct $K$-weights
are orthogonal. Contravariance then gives
\[
(v_r,v_r)
=(Yv_{r-1},v_r)
=(v_{r-1},Xv_r)
=\beta_r(v_{r-1},v_{r-1}),
\]
which proves \eqref{eq:smith-shapovalov-product} by induction.
Finally, substituting the potential \eqref{eq:potential}
and writing $\mu_r=\lambda-r+1$ gives
\eqref{eq:smith-shapovalov-factor}.
\end{proof}

\begin{proposition}
\label{prop:smith-finite-dimensional}
The highest-weight module $M(\lambda,c)$ has a
$d$-dimensional quotient if and only if
\[
\Phi(\lambda)=\Phi(\lambda-2d).
\]
This quotient is simple if and only if $d$ is the
smallest positive integer satisfying this equality.
Consequently, every finite-dimensional simple module
of the cubic Smith algebra is obtained in this way
for some $\lambda$ and $c$.
\end{proposition}

\begin{proof}
By Lemma~\ref{lem:smith-shapovalov},
\[
Xv_d
=\bigl[\Phi(\lambda)-\Phi(\lambda-2d)\bigr]v_{d-1}.
\]
Hence $v_d$ is a highest-weight vector precisely when
the displayed equality holds. In that case, the span
of $v_d,v_{d+1},\ldots$ is a submodule, and its quotient
has basis represented by $v_0,\ldots,v_{d-1}$.

If the equality also holds for some $1\leq j<d$,
the image of $v_j$ generates a proper nonzero submodule
of this quotient. Conversely, if no such $j$ exists,
every raising coefficient between $v_0$ and $v_{d-1}$
is nonzero. The weights are distinct, so every nonzero
submodule contains a weight vector and hence contains
all the basis vectors by applying $X$ and $Y$.
Thus the quotient is simple.

For completeness, let $V$ be a finite-dimensional simple
module. Choose an eigenvalue $\lambda$ of $K$ such that
$\lambda+2$ is not an eigenvalue; such a choice is possible
because the spectrum of $K$ is finite. For a nonzero
$\lambda$-eigenvector $v$, the relation $[K,X]=2X$ gives
\[
K(Xv)=(\lambda+2)Xv.
\]
Hence $Xv=0$. By simplicity, $v$ generates $V$, so $V$
is a simple quotient of $M(\lambda,c)$. This recovers
the exhaustiveness assertion of
Smith \cite[Lemma~2.3(a)]{S}.
\end{proof}

The minimality condition also has a description in terms
of the Shapovalov form. If $d$ is the first positive integer
for which $\beta_d=0$, then
\eqref{eq:smith-shapovalov-product} gives
\[
(v_r,v_r)\ne0\quad(0\leq r<d),\qquad
(v_r,v_r)=0\quad(r\geq d).
\]
Since distinct weight spaces are orthogonal, the radical
of the form is precisely
$\operatorname{span}\{v_d,v_{d+1},\ldots\}$.
The $d$-dimensional simple quotient is therefore the
quotient by this radical.

\begin{remark}
For the ordinary $\mathfrak{sl}_2$ highest-weight module,
the raising coefficient is
\[
e f^j v_0=j(\lambda-j+1)f^{j-1}v_0.
\]
As $j$ ranges over the positive integers, this coefficient
has at most one zero. If it vanishes, the vector at that
level generates the maximal proper submodule.
For a Smith highest-weight module, the corresponding
coefficient is
\[
\beta_j=\Phi(\lambda)-\Phi(\lambda-2j).
\]
It may vanish at several positive integers
$d_1<d_2<\cdots$. The resulting highest-weight vectors
give nested submodules
\[
Y^{d_1}M(\lambda,c)\supset
Y^{d_2}M(\lambda,c)\supset\cdots.
\]
The first of these is the maximal proper submodule,
so the simple quotient has dimension $d_1$.
The later zeros describe submodules within it.
This is why the finite-dimensional simplicity criterion
requires the \emph{smallest} positive zero, rather than
an arbitrary solution of
$\Phi(\lambda)=\Phi(\lambda-2d)$.
\end{remark}

The factorization \eqref{eq:smith-shapovalov-factor}
makes the finite-dimensional condition more explicit.
For a proposed dimension $d\geq1$, set
\[
\mu=\lambda-d+1.
\]
Then $\beta_d=0$ is equivalent to
\begin{equation}
\label{eq:smith-two-branches}
\mu=0
\qquad\text{or}\qquad
\mu^2+d^2=c+1.
\end{equation}
Thus finite-dimensional quotients have two candidate
branches in the $(\lambda,c)$-parameters. In either
branch, the quotient is simple precisely when
$\beta_j\ne0$ for every $1\leq j<d$.

\begin{proposition}[Range of the tensor-product construction]
\label{prop:smith-tensor-range}
Let $L$ be a finite-dimensional simple module for the
cubic Smith algebra, with dimension $d$, highest weight
$\lambda$, and central parameter $c$. Set
\[
\mu=\lambda-d+1.
\]
Then $L$ is isomorphic to the Smith module
$(L_n\otimes L_m)^+$ for some $n,m\in\mathbb Z_{\geq0}$
if and only if
\begin{equation}
\label{eq:smith-tensor-range}
d,\mu\in\mathbb Z_{>0},\qquad
\mu\geq d,\qquad
c+1=\mu^2+d^2.
\end{equation}
When these conditions hold, order the tensor factors so that
$n\geq m$. Then they are uniquely determined by
\[
n=\mu-1,\qquad m=d-1.
\]
Thus the tensor-product construction realizes precisely
the positive integral region $\mu\geq d$ of the second
branch in \eqref{eq:smith-two-branches}.
\end{proposition}

\begin{proof}
Suppose first that $L=(L_n\otimes L_m)^+$, and interchange
the factors if necessary so that $n\geq m$. By
Theorem~\ref{thm:smith-tensor-module},
\[
d=m+1,\qquad \lambda=n+m+1,\qquad
\mu=\lambda-d+1=n+1.
\]
In particular, $\mu\geq d$, and
\[
c+1
=n(n+2)+m(m+2)+2
=(n+1)^2+(m+1)^2
=\mu^2+d^2.
\]

Conversely, assume \eqref{eq:smith-tensor-range}, and set
$n=\mu-1$ and $m=d-1$. Then $n\geq m\geq0$.
Theorem~\ref{thm:smith-tensor-module} gives a Smith
module $(L_n\otimes L_m)^+$ of dimension $d$, highest
weight
\[
n+m+1=\mu+d-1=\lambda,
\]
and central parameter
\[
(n+1)^2+(m+1)^2-1
=\mu^2+d^2-1=c.
\]

To check the first-zero condition directly, substitute
these values into
\eqref{eq:smith-shapovalov-factor}. For $1\leq j\leq d$
one obtains
\[
\beta_j
=4j(\mu-j)(d-j)(\mu+d-j).
\]
Every factor is nonzero for $1\leq j<d$, whereas
$\beta_d=0$. Hence the corresponding $d$-dimensional
quotient of $M(\lambda,c)$ is simple by
Proposition~\ref{prop:smith-finite-dimensional}.
Both $L$ and $(L_n\otimes L_m)^+$ are simple quotients
with the same highest weight and central parameter,
so they are isomorphic.
\end{proof}

\begin{remark}
The first branch $\mu=0$ in \eqref{eq:smith-two-branches}
may be viewed as a Smith-algebra analogue of the
finite-dimensional irreducible $\mathfrak{sl}_2$-modules.
It lies outside the tensor-product range of
Proposition~\ref{prop:smith-tensor-range}.
In this branch, $\lambda=d-1$, so the $H$-weights are
$d-1,d-3,\ldots,1-d$, exactly as in the
$d$-dimensional irreducible $\mathfrak{sl}_2$-module.

With $v_r=F^rv_0$ for $0\leq r<d$, the action is
\[
\begin{aligned}
Ev_r&=r(d-r)A_rv_{r-1} &&(1\leq r\leq d-1),\\
Fv_r&=v_{r+1} &&(0\leq r\leq d-2),\\
Hv_r&=(d-1-2r)v_r &&(0\leq r\leq d-1),
\end{aligned}
\]
where
\[
A_r=2\bigl((d-r)^2+r^2-c-1\bigr),
\]
and $Ev_0=Fv_{d-1}=0$.
The module is simple precisely when
$A_r\neq0$ for $1\leq r<d$.

On each such simple module, define
\[
ev_0=0,\qquad
ev_r=A_r^{-1}Ev_r\quad(1\leq r<d),
\qquad f=F,\qquad h=H.
\]
Then $ev_r=r(d-r)v_{r-1}$, and these operators satisfy
\[
[h,e]=2e,\qquad [h,f]=-2f,\qquad [e,f]=h.
\]
Thus each simple module in the $\mu=0$ branch
admits the usual $d$-dimensional irreducible
$\mathfrak{sl}_2$ action after rescaling $E$
by the generally weight-dependent factor $A_r^{-1}$,
with $F$ and $H$ unchanged.
The original Smith action, however, retains its dependence
on $c$ through the coefficients $A_r$.
\end{remark}

\section{Conclusion}

We have shown that the diagonal Mickelsson--Zhelobenko
algebra associated with
$(\mathfrak{sl}_2\oplus\mathfrak{sl}_2,
\mathfrak{sl}_2^{\mathrm{diag}})$ contains a realization of
the cubic Smith algebra. The extremal projector first
constructs the singular vectors in $L_n\otimes L_m$;
the projected operators then connect the highest-weight
lines in its Clebsch--Gordan decomposition. After a
rescaling \eqref{eq:smith-generators}, 
these operators satisfy the cubic Smith
relation (Theorem~\ref{thm:cubic}), 
with central parameters determined by the
Casimir eigenvalues of the two tensor factors
(Proposition~\ref{prop:smith-central-parameters}).

Theorem~\ref{thm:smith-tensor-module} describes the
resulting Smith module explicitly. In particular, the
construction gives a simple module of dimension
$\min\{n,m\}+1$ on $(L_n\otimes L_m)^+$.
Proposition~\ref{prop:smith-tensor-range} determines
exactly which finite-dimensional simple Smith modules
arise in this way. Thus tensor-product branching provides
both a concrete representation of the cubic Smith algebra
and a precise description of the range of this
representation-theoretic construction.

\section*{Acknowledgment}

The authors are grateful to
Prof.\ Yoshiyuki Koga and Prof.\ Osami Yasukura
for valuable discussions.
They thank the organizers and participants of the
{\it $41$st Summer Seminar on Lie Algebras} and the
{\it $2026$ meeting of Dynamics of Fields and Spacetime
through Discrete Methods} for stimulating opportunities
to present and discuss this work.
In particular, they thank Prof.\ Ryo Sato for a stimulating
discussion on Smith algebras during the former seminar.
The work of TM was supported by
JSPS KAKENHI Grant Number JP24K06665.

%%%%%%%%%%%%%%%%%%%%%%%%%%%%%%%%%%
\appendix

%%%%%%%%%%%%%%%%%%%%%%%%%%%%%%%%%%
\appendix

\section{Explicit formula for the singular vectors}
\label{app:explicit}

The singular vectors obtained by the extremal-projector
construction admit an explicit closed formula.
Assume $n\geq m$ and write
\[
s_{N-2r}
=
\sum_{k=0}^{r}
c_{r,k}\,
f^{r-k}1_n\otimes f^k1_m,
\qquad N=n+m,\quad 0\leq r\leq m.
\]

\begin{proposition}
For $0\le r\le m$,
\begin{equation}
\label{eq:explicit}
c_{r,k}
=
(-1)^k\binom{r}{k}
\frac{(n-r+1)_k(m-r+1)_{r-k}}
     {(N-2r+2)_r},
\end{equation}
where $(a)_k=a(a+1)\cdots(a+k-1)$ is the rising Pochhammer symbol.
Equivalently,
\begin{equation}
\label{eq:binomial}
c_{r,k}
=
(-1)^k
\frac{
\binom{n-r+k}{k}\binom{m-k}{r-k}
}{
\binom{n+m-r+1}{r}
}.
\end{equation}
\end{proposition}

\begin{proof}
Using
\[
ef^a1_n=a(n-a+1)f^{a-1}1_n,
\]
the singularity condition $\Delta(e)s_{N-2r}=0$ gives
\[
\sum_{k=0}^{r-1}
\bigl[
c_{r,k}(r-k)(n-r+k+1)
+c_{r,k+1}(k+1)(m-k)
\bigr]
f^{r-k-1}1_n\otimes f^k1_m=0.
\]
By linear independence, we obtain the recurrence
\[
c_{r,k+1}
=-\frac{(r-k)(n-r+k+1)}{(k+1)(m-k)}c_{r,k}
\qquad (0\leq k<r).
\]
Solving it gives
\[
c_{r,k}
=(-1)^k
\frac{(-r)_k(n-r+1)_k}{k!\,(-m)_k}\,c_{r,0}
\qquad (0\leq k\leq r\leq m).
\]

To determine the common factor $c_{r,0}$,
use the functional $\sigma$ defined in
the proof of Proposition~\ref{prop:singular-basis}.
By the defining property of the extremal projection,
\[
f^r1_n\otimes1_m-s_{N-2r}
\in\Delta(f)V_{N-2r+2}.
\]
Since $\sigma$ vanishes on this subspace and
$\sigma(f^r1_n\otimes1_m)=1$, we obtain
$\sigma(s_{N-2r})=1$.
In the present notation, this means
\[
\sum_{k=0}^{r}(-1)^k c_{r,k}=1.
\]
Substituting the solution of the recurrence yields
\[
1=c_{r,0}
\sum_{k=0}^{r}
\frac{(-r)_k(n-r+1)_k}{k!\,(-m)_k}.
\]

We evaluate this sum using the Chu--Vandermonde identity
\cite[Eq.~(15.4.24)]{DLMF},
the terminating form of Gauss's hypergeometric
summation at unity:
\[
\sum_{k=0}^{r}\frac{(-r)_k(b)_k}{k!\,(a)_k}
=\frac{(a-b)_r}{(a)_r}.
\]
Here the identity is understood as a finite-sum identity;
the specialization $a=-m$ is valid because
$r\leq m$ ensures that all denominators are nonzero.
Taking $a=-m$ and $b=n-r+1$, we obtain
\[
\sum_{k=0}^{r}
\frac{(-r)_k(n-r+1)_k}{k!\,(-m)_k}
=\frac{(-N+r-1)_r}{(-m)_r}
=\frac{(N-2r+2)_r}{(m-r+1)_r}.
\]
In the last equality, we used the identity
$(-a)_k=(-1)^k(a-k+1)_k$.
Consequently,
\[
c_{r,0}
=\frac{(m-r+1)_r}{(N-2r+2)_r}.
\]

To simplify the resulting expression for $c_{r,k}$,
use
\[
(-r)_k=(-1)^k k!\binom{r}{k},\qquad
(-m)_k=(-1)^k(m-k+1)_k,
\]
together with
\[
(m-r+1)_r
=(m-r+1)_{r-k}(m-k+1)_k.
\]
Cancellation then gives
\[
c_{r,k}
=(-1)^k\binom{r}{k}
\frac{(n-r+1)_k(m-r+1)_{r-k}}
     {(N-2r+2)_r},
\]
which proves \eqref{eq:explicit}.

Finally, using the identity
\[
(a)_k=k!\binom{a+k-1}{k},
\]
we obtain
\[
c_{r,k}
=(-1)^k
\frac{
      \binom{n-r+k}{k}\binom{m-k}{r-k}}
     {\binom{N-r+1}{r}}. 
\]
This proves \eqref{eq:binomial}.

\end{proof}

\section{Explicit representatives of the MZ and Smith generators}
\label{app:MZgen}

We record explicit representatives of the individual
projected generators in
\[
Z(\mathfrak g,\mathfrak k)=\operatorname{Norm}(J)/J,
\qquad J=U'(\mathfrak g)E_{\mathrm d}.
\]
All rational functions of $D$ in the formulas below are
placed on the right. Equality here means equality of
classes modulo $J$.

For the raising generator, $e_1$ is already
$E_{\mathrm d}$-singular modulo $J$. Hence
\begin{align}
z^+=\Delta(p)e_1\Delta(p)=e_1=e\otimes1
\qquad \pmod J.
\end{align}

For the lowering generator, expanding the projector
against $f_1$ gives the representative
\begin{align}
z^-
&=\Delta(p)f_1\Delta(p) \notag\\
&=(f_1h_2-h_1f_2)\frac1D
  -F_{\mathrm d}^{\,2}e_1\frac1{D(D+1)}
  \notag \\
&=(f\otimes h-h\otimes f)\frac1D
  -(\Delta f)^{2}(e\otimes 1)\frac1{D(D+1)}
  \qquad \pmod J.
\label{eq:explicit-mz-minus}
\end{align}
Only the terms with projector index $k\leq2$
contribute modulo $J$, since $E_{\mathrm d}^3f_1\in J$.

One may also project the Cartan generator $h_1$.
Writing $z^0=\Delta(p)h_1\Delta(p)$, its class has the
equivalent forms
\begin{equation}
\label{eq:explicit-mz-zero}
z^0
=h_1+2F_{\mathrm d}e_1\frac1{D+2}
=\frac12\left(D+\frac{T_1-T_2}{D+2}\right)
\qquad \pmod J.
\end{equation}
The second equality follows from the definitions of
$T_1$ and $T_2$, after moving $E_{\mathrm d}$ to the
right and discarding terms in $J$.

These formulas represent individual elements of the MZ
algebra. Products such as $z^+z^-$ are computed in
$\operatorname{Norm}(J)/J$; their relations are established
in Section~\ref{sec:mz}.

In the normalization \eqref{eq:smith-generators}, the
corresponding Smith generators have representatives
\begin{align}
E&=e\otimes 1,\notag \\
F&=4\bigl[(f\otimes h-h\otimes f)(D+1)
          -(\Delta f)^{2}(e\otimes 1)\bigr],\\
H&=D+1=h\otimes 1+1\otimes h+1.
\notag 
\end{align}
In particular, although the displayed representative
of $z^-$ contains rational functions of $D$, the
denominators cancel in this representative of $F$.
The central parameter is $c=T_1+T_2+1$.

\section{Comparison with the quantum group
$U_q(\mathfrak{sl}_2)$}
\label{app:QG}

The relation between the Higgs algebra and the quantum
deformation of $SU(2)$ was studied by Zhedanov \cite{Zhed}.
Here we record how the deformation parameter in this
comparison relates formally to the tensor-product parameters
of our construction.

Write $q=e^\hbar$ and use the additive Cartan
presentation of $U_q(\mathfrak{sl}_2)$:
\[
[\widetilde H,\widetilde E]=2\widetilde E,\qquad
[\widetilde H,\widetilde F]=-2\widetilde F,\qquad
[\widetilde E,\widetilde F]
=\frac{q^{\widetilde H}-q^{-\widetilde H}}{q-q^{-1}}.
\]
Expanding the last relation at $\hbar=0$, 
with $\widetilde H$ held fixed, we have 
\[
\frac{q^{\widetilde H}-q^{-\widetilde H}}{q-q^{-1}}
=
\frac{\sinh(\hbar\widetilde H)}{\sinh\hbar}
 =
\widetilde H
+\frac{\hbar^2}{6}
  \bigl(\widetilde H^3-\widetilde H\bigr)
+O(\hbar^4).
\]
Set
\[
\widetilde E=\frac{\hbar}{2\sqrt3}E,\qquad
\widetilde F=\frac{\hbar}{2\sqrt3}F,\qquad
\widetilde H=H.
\]
This rescaling preserves the two Cartan relations,
while the commutator becomes
\[
[E,F]
=2\left(H^3-\left(1-\frac6{\hbar^2}\right)H\right)
+O(\hbar^2).
\]

Comparison with the cubic Smith relation
\[
[E,F]=2(H^3-cH)
\]
therefore gives the formal correspondence
\[
c=1-\frac{6}{\hbar^2}.
\]

In the tensor-product realization of Section~\ref{sec:smith},
the central parameter is $c=T_1+T_2+1$.
Consequently, the formal correspondence becomes
\[
\hbar^2=-\frac{6}{T_1+T_2}.
\]
On $L_n\otimes L_m$, where $T_1=n(n+2)$ and
$T_2=m(m+2)$, this reads
\[
\hbar^2
=-\frac{6}{n(n+2)+m(m+2)}
\]
whenever $(n,m)\ne(0,0)$.

The parameter relation also suggests a physical interpretation.
Since
\[
c-1=T_1+T_2=n(n+2)+m(m+2),
\]
the parameter correspondence gives
$|\hbar|\sim\sqrt{6/c}$ as the spins grow.
Thus increasing the spins corresponds to decreasing
$|\hbar|$ at the level of these parameters.
However, the expansion above is taken with $H$ fixed,
whereas the $H$-weights also grow in the tensor-product
construction.
Small $|\hbar|$ therefore does not by itself justify
truncating the commutator expansion on these modules.

The comparison concerns only the formal matching
of coefficients in the defining commutation relation.
It neither establishes a large-spin or semiclassical
limit nor provides an action of
$U_q(\mathfrak{sl}_2)$ on the singular subspace.
A representation-theoretic interpretation of this
parameter correspondence remains an open question.

\end{document}